\documentclass[12pt,reqno]{article}

\usepackage{xcolor}
\usepackage{amssymb}
\usepackage{amsmath}
\usepackage{amsthm}
\usepackage{amsfonts} 
\usepackage{amscd}
\usepackage{graphicx}

\usepackage[colorlinks=true,
linkcolor=webgreen,
filecolor=webbrown,
citecolor=webgreen]{hyperref}

\definecolor{webgreen}{rgb}{0,.5,0}
\definecolor{webbrown}{rgb}{.6,0,0}

\usepackage{color}
\usepackage{fullpage}
\usepackage{float}

\usepackage{latexsym}
\usepackage{accents}

\usepackage{graphics}
\numberwithin{equation}{section}

\begin{document}

\theoremstyle{plain}
\newtheorem{theorem}{Theorem}
\newtheorem{corollary}[theorem]{Corollary}
\newtheorem{lemma}[theorem]{Lemma}
\newtheorem{remark}{Remark}
\newtheorem{example}{Example}

\begin{center}
{\LARGE\bf
On Some Series Involving the Binomial Coefficients $\binom{3n}{n}$ \\}

\vskip 1.4cm

{\large

Kunle Adegoke \\
Department of Physics and Engineering Physics \\ Obafemi Awolowo University\\
Ile-Ife, Nigeria \\
\href{mailto:adegoke00@gmail.com}{\tt adegoke00@gmail.com}

\vskip 0.25 in

Robert Frontczak\\
Independent Researcher\\
Reutlingen, Germany \\
\href{mailto:robert.frontczak@web.de}{\tt robert.frontczak@web.de}

\vskip 0.25 in

Taras Goy  \\
Faculty of Mathematics and Computer Science\\
Vasyl Stefanyk Precarpathian National University\\
Ivano-Frankivsk, Ukraine\\
\href{mailto:taras.goy@pnu.edu.ua}{\tt taras.goy@pnu.edu.ua}}

\end{center}

\vskip .25 in
\begin{center}
	{\it Dedicated to the memory of Professor Khristo N. Boyadzhiev \\ who passed away on June 28, 2023, and who loved infinite series.}
\end{center}
\vskip .25 in
\begin{abstract}
Using a simple transformation, we obtain much simpler forms for some series involving binomial coefficients $\binom{3n}n$ derived by Necdet Batir. New evaluations are given; and connections with Fibonacci numbers and the golden ratio are established. Finally, we derive some Fibonacci and Lucas series involving the reciprocals of $\binom{3n}n$.
\vskip 10pt
\noindent\textit{2020 Mathematics Subject Classification}: 40A05, 11B65, 11B39.
\vskip 5 pt
\noindent\textit{Keywords}: series; binomial coefficient; binomial sum; Fibonacci number; Lucas number. 
\end{abstract}

\newpage

\section{Introduction}
In an article published in the year 2005,  Batir \cite{batir03}, inspired by the results of Lehmer \cite{lehmer85}, studied the series $\sum\limits_{k = 1}^\infty  \frac{z^k}{k^{n}\binom{3n}n}$, giving particular attention to the special cases $n\in \mathbb{N}\cup\{0\}$, for which he derived explicit closed formulas. He obtained many interesting formulas by evaluating the closed forms at appropriate arguments. Some of his results had earlier been obtained experimentally by Borwein and Girgensohn \cite{borwein05}. In \cite{DAurizio}, D'Aurizio and Di~Trani studied this kind of series using hypergeometric functions.
In the recent paper \cite{chu}, Chu evaluated many series having the form $\sum\limits_{k = 1}^\infty  \frac{z^k}{k^{a+1}\binom{3n+b}n}$, where $a\in\{0;\pm1;\pm2\}$ and $b\in\{0;1;\pm2\}$. 

The purpose of this note is to derive equivalent but much simpler expressions for the special cases and thereby obtain new evaluations.

Batir \cite[Identity (3.1)]{batir03} showed, for $|z|\leq{27}/{4}$, that
\begin{equation} \label{Batir}
\sum_{k = 1}^\infty  \frac{z^k}{k^2 \binom{3k}k}  = 6\arctan ^2\!\bigg( {\frac{\sqrt 3 }{{2\phi(z)  - 1}}} \bigg) - \frac{1}{2}\log^2\!\bigg( \frac{{\phi^3(z)  + 1}}{\big(\phi(z) + 1\big)^3 } \bigg),
\end{equation}
where
\begin{equation}\label{phi}
\phi (z) = \sqrt[3]{\frac{{27 - 2z + 3\sqrt{81 - 12z} }}{2z}}.
\end{equation}

At $z={27}/{4}$, $z={20}/{3}$, $z={77}/{12}$, $z=6$, $z={65}/{12}$, $z={14}/{3}$,  $z={15}/{4}$,  $z={8}/{3}$, and $z={17}/{12}$ (then the expression $81-12z$ in \eqref{phi} will be a perfect square), from \eqref{Batir} we immediately obtain, respectively, such series: 
\begin{align}
\sum_{k=1}^\infty \frac{\left(\frac{27}{4}\right)^k}{k^2\binom{3k}{k}}&= \frac{2\pi^2}{3}-2\log^2 2,\label{27/4}\\
\sum_{k=1}^\infty \frac{\left(\frac{20}{3}\right)^k}{k^2\binom{3k}{k}}&=6\arctan^2\!\bigg(\frac{\sqrt3}{\sqrt[3]{10}-1}\bigg)-\frac12\log^2\!\bigg(\frac{18}{\big(\sqrt[3]{10}+2\big)^3}\bigg) ,\nonumber\\
\sum_{k=1}^\infty \frac{\left(\frac{77}{12}\right)^k}{k^2\binom{3k}{k}}&=6\arctan^2\!\bigg(\frac{7\sqrt3}{2\sqrt[3]{539}-7}\bigg)-\frac12\log^2\!\bigg(\frac{882}{\big(\sqrt[3]{539}+7\big)^3}\bigg),\nonumber\\
\sum_{k=1}^\infty\frac{6^k}{k^2\binom{3k}{k}}&= 6\arctan^2\!\bigg(\frac{\sqrt3}{2\sqrt[3]{2}-1}\bigg)-\frac12\log^2\!\bigg(\frac{3}{\big(\sqrt[3]{2}+1\big)^3}\bigg),\label{6}\\
\sum_{k=1}^\infty \frac{\left(\frac{65}{12}\right)^k}{k^2\binom{3k}{k}}&= 6\arctan^2\!\bigg(\frac{5\sqrt3}{2\sqrt[3]{325}-5}\bigg)-\frac12\log^2\!\bigg(\frac{450}{\big(\sqrt[3]{325}+5\big)^3}\bigg),\nonumber\\
\sum_{k=1}^\infty \frac{\left(\frac{14}{3}\right)^k}{k^2\binom{3k}{k}}&= 6\arctan^2\!\bigg(\frac{\sqrt3}{\sqrt[3]{28}-1}\bigg)-\frac12\log^2\!\bigg(\frac{36}{\big(\sqrt[3]{28}+2\big)^3}\bigg),\nonumber\\
\sum_{k=1}^\infty \frac{\left(\frac{15}{4}\right)^k}{k^2\binom{3k}{k}}&= 6\arctan^2\!\bigg(\frac{\sqrt3}{2\sqrt[3]{5}-1}\bigg)-\frac12\log^2\!\bigg(\frac{6}{\big(\sqrt[3]{5}+1\big)^3}\bigg),\nonumber
\end{align}
\begin{align}
\sum_{k=1}^\infty \frac{\left(\frac{8}{3}\right)^k}{k^2\binom{3k}{k}}&= \frac{\pi^2}{6}-\frac{\log^2 3}2,\label{Italy}\\
\sum_{k=1}^\infty \frac{\left(\frac{17}{12}\right)^k}{k^2\binom{3k}{k}}&= 6\arctan^2\!\bigg(\frac{\sqrt3}{2\sqrt[3]{17}-1}\bigg)-\frac12\log^2\!\bigg(\frac{18}{\big(\sqrt[3]{17}+1\big)^3}\bigg).\nonumber
\end{align}

Series \eqref{27/4} and \eqref{6} one  can find in \cite[Identities (3.4) and (3.5)]{batir03}. Series \eqref{Italy} was obtained by D'Aurizio and Di Trani using the hypergeometric function ${}_4F_3$ \cite[Formula (8)]{DAurizio}.  

Similarly, we find the corresponding alternating series:
\begin{align*}
\sum_{k=1}^\infty \frac{\left(-\frac{27}{4}\right)^k}{k^2\binom{3k}{k}}&= 6\arctan^2\!\bigg(\frac{\sqrt3}{2\sqrt[3]{3+2\sqrt2}+1}\bigg)-\frac12 \log^2\!\bigg(\frac{2+2\sqrt2}{\big(\sqrt[3]{3+2\sqrt2}-1\big)^3}\bigg),\\
\sum_{k=1}^\infty  \frac{\left(-\frac{20}{3}\right)^k}{k^2\binom{3k}{k}}&=6\arctan^2\!\bigg(\frac{\sqrt3\sqrt[3]{40}}{2\sqrt[3]{121+9\sqrt{161}}+\sqrt[3]{40}}\bigg)-\frac12 \log^2\!\bigg(\frac{81+9\sqrt{161}}{\big(\sqrt[3]{121+9\sqrt{161}}-\sqrt[3]{40}\big)^3}\bigg),\\
\sum_{k=1}^\infty  \frac{\left(-\frac{77}{12}\right)^k}{k^2\binom{3k}{k}}&= 6\arctan^2\!\bigg(\frac{\sqrt3\sqrt[3]{77}}{2\sqrt[3]{239+36\sqrt{158}}+\sqrt[3]{77}}\bigg)-\frac12 \log^2\!\bigg(\frac{162+18\sqrt{158}}{\big(\sqrt[3]{239+18\sqrt{158}}-\sqrt[3]{77}\big)^3}\bigg),\\
\sum_{k=1}^\infty\frac{(-6)^k}{k^2\binom{3k}{k}}&=  6\arctan^2\!\bigg(\frac{\sqrt3}{\sqrt[3]{26+6\sqrt{17}}+1}\bigg)-\frac12 \log^2\!\bigg(\frac{9+3\sqrt{17}}{\big(\sqrt[3]{13+3\sqrt{17}}-\sqrt[3]{4}\big)^3}\bigg),\\
\sum_{k=1}^\infty  \frac{\left(-\frac{65}{12}\right)^k}{k^2\binom{3k}{k}}&= 6\arctan^2\!\bigg(\frac{\sqrt3\sqrt[3]{65}}{2\sqrt[3]{227+18\sqrt{146}}+\sqrt[3]{65}}\bigg)-\frac12 \log^2\!\bigg(\frac{162+18\sqrt{146}}{\big(\sqrt[3]{227+18\sqrt{146}}-\sqrt[3]{65}\big)^3}\bigg),\\
\sum_{k=1}^\infty  \frac{\left(-\frac{14}{3}\right)^k}{k^2\binom{3k}{k}}&= 6\arctan^2\!\bigg(\frac{\sqrt3\sqrt[3]{7}}{\sqrt[3]{218+18\sqrt{137}}+\sqrt[3]{7}}\bigg)-\frac12 \log^2\!\bigg(\frac{81+9\sqrt{137}}{\big(\sqrt[3]{109+9\sqrt{137}}-\sqrt[3]{28}\big)^3}\bigg),\\
\sum_{k=1}^\infty \frac{\left(-\frac{15}{4}\right)^k}{k^2\binom{3k}{k}}&=  6\arctan^2\!\bigg(\frac{\sqrt3\sqrt[3]{5}}{2\sqrt[3]{23+6\sqrt{14}}+\sqrt[3]{5}}\bigg)-\frac12 \log^2\!\bigg(\frac{18+6\sqrt{14}}{\big(\sqrt[3]{23+6\sqrt{14}}-\sqrt[3]{5}\big)^3}\bigg),\\
\sum_{k=1}^\infty \frac{\left(-\frac{8}{3}\right)^k}{k^2\binom{3k}{k}}&= 
6\arctan^2\!\bigg(\frac{\sqrt3\sqrt[3]{2}}{\sqrt[3]{97+9\sqrt{113}}+\sqrt[3]{2}}\bigg)-\frac12 \log^2\!\bigg(\frac{81+9\sqrt{113}}{\big(\sqrt[3]{97+9\sqrt{113}}-\sqrt[3]{16}\big)^3}\bigg),\\
\sum_{k=1}^\infty \frac{\left(-\frac{17}{12}\right)^k}{k^2\binom{3k}{k}}&= 6\arctan^2\!\bigg(\frac{\sqrt3\sqrt[3]{17}}{2\sqrt[3]{179+126\sqrt{2}}+\sqrt[3]{17}}\bigg)-\frac12 \log^2\!\bigg(\frac{162+126\sqrt{2}}{\big(\sqrt[3]{179+126\sqrt{2}}-\sqrt[3]{17}\big)^3}\bigg).
\end{align*}

The substitution $z=\frac{27xy}{(x+y)^{2}}$ in \eqref{Batir} reduces function $\phi(z(x,y))$ to $\sqrt[3]{{x}/{y}}$, thereby yielding the following more manageable identity:
\begin{equation}\label{A}\tag{A}
\sum_{k = 1}^\infty  \frac{(27xy)^k }{{k^2 (x + y)^{2k} \binom{3k}k}}  = 
6\arctan ^2\! \bigg( \frac{\sqrt 3\sqrt[3]{y} }{2\sqrt[3]{x}  - \sqrt[3]{y} } \bigg) - \frac{1}{2}\log ^2\! \bigg( {\frac{{x + y}}{{(\sqrt[3]{x}  + \sqrt[3]{y} )^3 }}} \bigg),
\end{equation}
which is valid for ${x}/{y}\ge 1$ or ${x}/{y}\le -(\sqrt2+1)^2=-\cot^2(\pi/8)$. 

Differentiating twice identity \eqref{A} with respect to $x$, we obtain, for $x/y>1$ or $x/y\le -(\sqrt2+1)^2$, the following identities:
\begin{equation}\label{B}\tag{B}
\begin{split}
&\hspace{-1.9cm}\sum_{k = 1}^\infty  {\frac{(27xy)^k }{{k(x + y)^{2k} \binom{3k}k}}}  \\
\,\,\,\,\,\,\,\,= \frac{\sqrt[3]{xy} }{{x - y}} & \left(2\sqrt 3\big(\sqrt[3]{x}  + \sqrt[3]{y} \big)\arctan\! \bigg( {\frac{\sqrt 3\sqrt[3]{y}}{2\sqrt[3]{x}  - \sqrt[3]{y} }} \bigg)+ \big(\sqrt[3]{x}  - \sqrt[3]{y} \big)\log\! \bigg( \frac{{x + y}}{{\big(\sqrt[3]{x}  + \sqrt[3]{y} \big)^3 }} \bigg)\!\right)
\end{split}
\end{equation}
and 
\begin{equation}\label{C}\tag{C}
\begin{split}
&\hspace{-0.5cm}\sum_{k = 1}^\infty\frac{(27xy)^k}{(x + y)^{2k} \binom{3k}k} = \frac{4xy}{(x - y)^2}\\
&\,\,\,\,\,+\frac{\sqrt[3]{xy}}{3}\frac{x + y}{(x - y)^3 }\left( 
2\sqrt 3\left( {2\sqrt[3]{xy} \big(\sqrt[3]{x^2}  + \sqrt[3]{y^2} \big) + \sqrt[3]{x^4}  + \sqrt[3]{y^4}} \right)\arctan\! \bigg( \frac{\sqrt 3\sqrt[3]{y}}{2\sqrt[3]{x}- \sqrt[3]{y} } \bigg)\right. \\
&\,\,\,\,\,\left.- \left( {2\sqrt[3]{xy} \big(\sqrt[3]{x^2}  - \sqrt[3]{y^2} \big) - \sqrt[3]{x^4}  + \sqrt[3]{y^4}} \right)\log\! \bigg( {\frac{{x + y}}{\big(\sqrt[3]{x}  + \sqrt[3]{y} \big)^3 }} \bigg)\!\right)\!.
\end{split}
\end{equation}

Setting $(x,y)=(8,1)$, $(x,y)=(8,-1)$, $(x,y)=(8,1/8)$, $(x,y)=(8,-1/8)$, $(x,y)=(1,1/{27})$ and $(x,y)=(1,-1/{27})$ in \eqref{A}, \eqref{B}, and \eqref{C} we have the following series list:
\begin{align*}
\sum_{k = 1}^\infty  \frac{ \big(\frac{8}{3}\big)^{k}}{k \binom{3k}k}  &= \frac{2\sqrt{3}\,\pi}{7}-\frac{2}{7}\log3,  
\\
\sum_{k = 1}^\infty  \frac{ \big(\frac{8}{3}\big)^{k}}{\binom{3k}k}  &= \frac{32}{49}+\frac{74\sqrt3\,\pi}{343}-\frac{18}{343}\log3,\\
\sum_{k = 1}^\infty  (-1)^k\frac{ \big(\frac{6\sqrt6}{7}\big)^{2k}}{k^2 \binom{3k}k}  &= 6\arctan ^2 \Bigl( {\frac{{\sqrt 3 }}{5}} \Bigr) - \frac{1}{2}\log ^2 7, \\
\sum_{k = 1}^\infty  (-1)^k\frac{ \big(\frac{6\sqrt6}{7}\big)^{2k}}{{k\binom{3k}k}} & = \frac{{4\sqrt 3 }}{9}\arctan \Bigl( {\frac{{\sqrt 3 }}{5}} \Bigr) - \frac{2}{3}\log 7,
\\
\sum_{k = 1}^\infty (-1)^{k}{\frac{ \big(\frac{6\sqrt6}{7}\big)^{2k} }{\binom{3k}k}} & =   -\frac{32}{81} - \frac{28\sqrt 3}{729}\arctan \Bigl( {\frac{{\sqrt 3 }}{5}} \Bigr) - \frac{14}{81}\log 7,
\\
\sum_{k = 1}^\infty  {\frac{{\big(\frac{24\sqrt3}{65}\big)^{2k} }}{{k^2 \binom{3k}k}}}  &= 6 \arctan ^2 \Bigl( {\frac{{\sqrt 3 }}{7}} \Bigr) - \frac12\log ^2\Bigl(\frac{25}{13}\Bigr),
\\
\sum_{k = 1}^\infty  \frac{{\big(\frac{24\sqrt3}{65}\big)^{2k} }}{{k \binom{3k}k}}  &= \frac{40\sqrt3}{63}\arctan\Bigl( \frac{{\sqrt 3 }}{7} \Bigr) - \frac{4}{21}\log\Bigl(\frac{25}{13}\Bigr),\\
\sum_{k = 1}^\infty  \frac{{\big(\frac{24\sqrt3}{65}\big)^{2k} }}{\binom{3k}k}  &=\frac{256}{3969} + \frac{68120\sqrt3}{250047}\arctan  \Bigl( {\frac{{\sqrt 3 }}{7}} \Bigr) - \frac{1300}{27783}\log\Bigl(\frac{25}{13}\Bigr),\end{align*}
\begin{align*}
\sum_{k = 1}^\infty (-1)^k \frac{\big(\frac{8\sqrt3}{21}\big)^{2k} }{k^2\binom{3k}k}  &= 6\arctan^2  \Bigl( {\frac{{\sqrt 3 }}{9}} \Bigr) - \frac{1}{2}\log^2\Bigl(\frac{7}{3}\Bigr),
\\
\sum_{k = 1}^\infty (-1)^k \frac{\big(\frac{8\sqrt3}{21}\big)^{2k} }{k\binom{3k}k}  &= \frac{24\sqrt3}{65}\arctan  \Bigl( \frac{\sqrt 3}{9} \Bigr) - \frac{4}{13}\log\Bigl(\frac{7}{3}\Bigr),\\
\sum_{k = 1}^\infty (-1)^{k} \frac{\big(\frac{8\sqrt3}{21}\big)^{2k} }{\binom{3k}k}  &=-\frac{256}{4225}+ \frac{20328\sqrt3}{274625}\arctan\Bigl( {\frac{\sqrt 3 }{9}} \Bigr) - \frac{252}{2197}\log\Bigl(\frac{7}{3}\Bigr),\\
\sum_{k = 1}^\infty  \frac{\big(\frac{27}{28}\big)^{2k} }{k^2\binom{3k}k}  &= 6\arctan^2  \Bigl( {\frac{{\sqrt 3 }}{5}} \Bigr) - \frac{1}{2}\log^2\Bigl(\frac{16}{7}\Bigr),
\\
\sum_{k = 1}^\infty  {\frac{{\left(\frac{27}{28}\right)^{2k} }}{k\binom{3k}k}}  &= \frac{12\sqrt3}{13}\arctan  \Bigl( {\frac{{\sqrt 3 }}{5}} \Bigr) - \frac{3}{13}\log\Bigl(\frac{16}{7}\Bigr),
\\
\sum_{k = 1}^\infty  {\frac{{\left(\frac{27}{28}\right)^{2k} }}{\binom{3k}k}}  &=\frac{27}{169} + \frac{994\sqrt3}{2197}\arctan  \Bigl( {\frac{{\sqrt 3 }}{5}} \Bigr) - \frac{112}{2197}\log\Bigl(\frac{16}{7}\Bigr),
\\
\sum_{k = 1}^\infty (-1)^k \frac{\big(\frac{27}{26}\big)^{2k} }{k^2\binom{3k}k}  &= 6\arctan^2  \Bigl( {\frac{{\sqrt 3 }}{7}} \Bigr) - \frac{1}{2}\log^2\Bigl(\frac{13}{4}\Bigr),
\\
\sum_{k = 1}^\infty (-1)^k {\frac{{\left(\frac{27}{26}\right)^{2k} }}{k\binom{3k}k}}  &= \frac{3\sqrt3}{7}\arctan  \Bigl( {\frac{{\sqrt 3 }}{7}} \Bigr) - \frac{3}{7}\log\Bigl(\frac{13}{4}\Bigr),
\\
\sum_{k = 1}^\infty (-1)^{k} {\frac{{\left(\frac{27}{26}\right)^{2k} }}{\binom{3k}k}}  &=-\frac{27}{196} + \frac{143\sqrt3}{2744}\arctan  \Bigl( {\frac{{\sqrt 3 }}{7}} \Bigr) - \frac{52}{343}\log\Bigl(\frac{13}{4}\Bigr),
\\
\sum_{k = 1}^\infty  \frac{\big(\frac{54\sqrt2}{35}\big)^{2k} }{k^2\binom{3k}k}  &= 6\arctan^2  \Bigl( {\frac{\sqrt 3 }{2}} \Bigr) - \frac{1}{2}\log^2\Bigl(\frac{25}{7}\Bigr),
\\
\sum_{k = 1}^\infty  \frac{\big(\frac{54\sqrt2}{35}\big)^{2k} }{k\binom{3k}k}  &= \frac{60\sqrt3}{19}\arctan\Bigl( \frac{\sqrt 3 }{2} \Bigr) - \frac{6}{19}\log\Bigl(\frac{25}{7}\Bigr),
\\
\sum_{k = 1}^\infty  \frac{\big(\frac{54\sqrt2}{35}\big)^{2k} }{\binom{3k}k}  &= \frac{864}{361}+\frac{35420\sqrt3}{6859}\arctan\Bigl( \frac{\sqrt 3}{2} \Bigr) - \frac{350}{6859}\log\Bigl(\frac{25}{7}\Bigr),
\\
\sum_{k = 1}^\infty  (-1)^k\frac{\big(\frac{54\sqrt2}{19}\big)^{2k} }{k^2\binom{3k}k}  &= 6\arctan^2  \Bigl( {\frac{\sqrt 3 }{4}} \Bigr) - \frac{1}{2}\log^2 19,
\\
\sum_{k = 1}^\infty  (-1)^k\frac{\big(\frac{54\sqrt2}{19}\big)^{2k} }{k\binom{3k}k}  &= \frac{12\sqrt3}{35}\arctan\Bigl( {\frac{\sqrt3}{4}} \Bigr) - \frac{6}{7}\log19,
\\
\sum_{k = 1}^\infty  (-1)^k\frac{\big(\frac{54\sqrt2}{19}\big)^{2k} }{\binom{3k}k}  &= -\frac{864}{1225}-\frac{4484\sqrt3}{42875}\arctan\Bigl( {\frac{\sqrt3}{4}} \Bigr) - \frac{38}{343}\log19.
\end{align*}

The first two series from this list can be found in \cite[Corollaries  2.3 and 3.3]{chu}.

\section{Evaluations at selected arguments}

In this section we will evaluate identities \eqref{A}, \eqref{B} and \eqref{C} at carefully selected values of $x$ and $y$. Some of the resulting summation identities will involve Fibonacci and Lucas numbers in the summand and possibly in the evaluations. 

Let $F_n$ and $L_n$ denote the $n$-th Fibonacci and Lucas numbers, both satisfying the recurrence relation
$$X_n = X_{n-1} + X_{n-2},\quad n\geq 2,$$ but with the initial conditions $F_0 = 0$, $F_1 = 1$ and $L_0 = 2$, $L_1 = 1$.
Extending these numbers to negative subscripts gives 
$$F_{-j}=(-1)^{j-1}F_j,\quad L_{-j}=(-1)^jL_j.$$

Throughout the paper, we denote the golden ratio  $\alpha=(1+\sqrt 5)/2$ and write $\beta=(1-\sqrt 5)/2$, so that $\alpha\beta=-1$ and $\alpha+\beta=1$. For any integer $j$, the explicit formulas (Binet formulas) for Fibonacci and Lucas numbers are
\begin{equation}\label{binet}
F_j  = \frac{{\alpha ^j  - \beta ^j }}{{\alpha  - \beta }},\qquad L_j  = \alpha ^j  + \beta ^j.
\end{equation}

We will often require the following identities, valid for any integer $r$, which are straightforward consequences of \eqref{binet}:
\begin{align}
\label{F1}
\alpha ^{2r}  + ( - 1)^{r + 1}  &= \alpha ^r F_r \sqrt 5 ,\\[3pt]
\label{F2}
\alpha ^{2r}  + ( - 1)^r  &= \alpha ^r L_r,\\[3pt]
\beta ^{2r}  + ( - 1)^{r + 1} & =  - \beta ^r F_r \sqrt 5,\\[3pt]
\beta ^{2r}  + ( - 1)^r  &= \beta ^r L_r.
\end{align}

We also require the following well-known identities \cite{koshy,vajda}:
\begin{align}
\label{F3} F_n^2  + ( - 1)^{n + m - 1} F_m^2  &= F_{n - m} F_{n + m}, \\[3pt]
 \label{F4} F_{n + m}  + ( - 1)^m F_{n - m}  &= L_m F_n ,\\[3pt]
\label{F5} F_{n + m}  + ( - 1)^{m - 1} F_{n - m} & = F_m L_n,\\[3pt]
\label{F6} L_n F_m  + F_n L_m  &= 2F_{n + m},\\[3pt] 
\label{F7} L_{n + m}  + ( - 1)^m L_{n - m} & = L_m L_n ,\\[3pt]
\label{F8} L_{n + m}  + ( - 1)^{m - 1} L_{n - m}  &= 5F_m F_n .
\end{align}

\subsection{Results from identity \eqref{A}}

\begin{theorem}\label{Th1}
	If $r$ is a non-negative integer, then
	\begin{align}
		\begin{split}\label{eq.mbx7fws}
	\sum_{k = 1}^\infty  {\frac{(-1)^{k(r-1)} \big(\frac{27}{5}\big)^k}{k^2\binom{3k}k F_r^{2k} }} &\\
	 = 6\arctan\! ^2 &\bigg(\frac{{\sqrt 3 }}{2\sqrt[3]{\alpha^{2r}}  + (-1)^{r}} \bigg) - \frac{1}{2}\log ^2\! \bigg( {\frac{\sqrt 5\alpha ^r F_r }{\big(\sqrt[3]{\alpha^{2r}}  - ( - 1)^{r} \big)^3 }} \bigg),\quad r\ne0,
	\end{split}\\
		\begin{split}
	\label{eq.w7owcwt}
	\sum_{k = 1}^\infty  {\frac{{( - 1)^{kr} 27^k }}{{k^2 \binom{3k}k}L_r^{2k} }} \hspace{1cm}& \\
	= 6\arctan ^2 &\bigg( {\frac{{ \sqrt 3 }}{{2\sqrt[3]{\alpha^{2r}}  - ( - 1)^r }}} \bigg) - \frac{1}{2}\log ^2\!\bigg( {\frac{{\alpha ^r L_r }}{{\big(\sqrt[3]{\alpha^{2r}}  + ( - 1)^r \big)^3 }}} \bigg),\quad r\ne 1.
	\end{split}
	\end{align}
\end{theorem}
\begin{proof}
	Identity \eqref{eq.mbx7fws} is proved by setting $x=\alpha^{2r}$, $y=(-1)^{r + 1}$ in \eqref{A} and making use of identity \eqref{F1}. Identity \eqref{eq.w7owcwt} follows from setting $x=\alpha^{2r}$, $y=(-1)^r$ in \eqref{A} and using  \eqref{F2}.
\end{proof}
\begin{example} Evaluation at $r=1$, $r=2$ and $r=3$ in \eqref{eq.mbx7fws} and \eqref{eq.w7owcwt}, respectively, gives 
	\begin{align*}
	\sum_{k = 1}^\infty  {\frac{\big(\frac{27}5\big)^k }{{k^2 \binom{3k}k}}}  &= 6\arctan ^2\!\bigg( {\frac{{\sqrt 3 }}{{2\sqrt[3]{\alpha^{2}}  - 1}}} \bigg) - \frac{1}{2}\log ^2\! \bigg( {\frac{{\alpha \sqrt 5 }}{{\big(\sqrt[3]{\alpha^{2}}  + 1\big)^3 }}} \bigg),\\
	\sum_{k = 1}^\infty  {\frac{( - 27)^{k}}{{k^2 \binom{3k}k}}}  &=  6\arctan ^2\! \bigg( {\frac{{\sqrt 3 }}{{2\sqrt[3]{\alpha^{2}}  + 1}}} \bigg) - \frac{1}{2}\log ^2\! \bigg( {\frac{\alpha }{{\big(\sqrt[3]{\alpha^{2}}  - 1\big)^3 }}} \bigg),\\
	\sum_{k = 1}^\infty  {\frac{{\left(-\frac{27}{5}\right)^k }}{{k^2 \binom{3k}k}}} & =   6\arctan ^2\! \bigg( {\frac{{\sqrt 3 }}{{2\sqrt[3]{\alpha^{4}}  + 1}}} \bigg) - \frac{1}{2}\log ^2\! \bigg( \frac{{\alpha ^2 \sqrt 5 }}{{\big(\sqrt[3]{\alpha^{4}}  - 1\big)^3 }} \bigg),\\
	\sum_{k = 1}^\infty  {\frac{{3^k }}{{k^2 \binom{3k}k}}}  &= 6\arctan ^2\! \bigg( {\frac{{\sqrt 3 }}{{2\sqrt[3]{\alpha^{4}}  - 1}}} \bigg) - \frac{1}{2}\log ^2\! \bigg( {\frac{{3\alpha ^2 }}{{\big(\sqrt[3]{\alpha^{4}}  + 1\big)^3 }}} \bigg),\\
	\sum_{k =1 1}^\infty  {\frac{{\big(\frac{27}{20}\big)^k }}{{k^2 \binom{3k}k}}}  &= 6\arctan ^2\! \big(\sqrt{15} -\sqrt{12}\big) - \frac{1}{2}\log ^2\! \Big(\frac52\Big),\\
	\sum_{k = 1}^\infty  {\frac{ \left(-\frac{27}{16}\right)^k }{{k^2 \binom{3k}k}}} & =  6\arctan ^2\! \Bigl({\frac{4\sqrt 3-\sqrt{15}}{11}} \Bigr) -  2\log ^2 2.
	\end{align*}
\end{example} 
\begin{corollary}
	If $r$ is a non-negative integer, then
	\begin{align*}
	\sum_{k = 1}^\infty  {\frac{{(-1)^{k(r - 1)} \big(\frac{27}{5}\big)^k }}{{k^2 \binom{3k}k} F_{3r}^{2k} }} & = 6\arctan ^2\! \bigg( {\frac{{\sqrt 3 }}{{\alpha ^{2r} + \alpha^{r}L_r }}} \bigg) - \frac{1}{2}\log ^2\! \left( {\frac{{F_{3r} }}{{5F_r^3 }}} \right)\!,\quad r\ne0,\\
	\sum_{k = 1}^\infty  {\frac{{( - 1)^{kr} 27^k }}{k^2  \binom{3k}kL_{3r}^{2k}}}  &= 6\arctan ^2\! \bigg( {\frac{{\sqrt 3 }}{{\alpha ^{2r}  + \sqrt 5\alpha^r F_r  }}} \bigg) - \frac{1}{2}\log ^2\! \left( {\frac{{L_{3r} }}{{L_r^3 }}} \right)\!.
	\end{align*}
\end{corollary}
\begin{proof} 	Replace $r$ with $3r$ in \eqref{eq.mbx7fws} and \eqref{eq.w7owcwt} and use \eqref{F1}, \eqref{F2}.
\end{proof}
\begin{theorem}
	Let $m$ and $n$ be positive integers such that $n \geq m$ unless stated otherwise. Then
	\begin{align}
	\begin{split}
	\sum_{k = 1}^\infty & \frac{( - 1)^{k(n - m - 1)}}{k^2 \binom{3k}k} \bigg( {\frac{{27F_n^2 F_m^2 }}{{F_{n - m}^2 F_{n + m}^2 }}} \bigg)^k\nonumber\\ 
	& = 6\arctan ^2\! \bigg( \frac{\sqrt 3\sqrt[3]{F_m^2}}{2 \sqrt[3]{F_n^2}  + ( - 1)^{n - m} \sqrt[3]{F_m^2}}\bigg) - \frac{1}{2}\log ^2\! \bigg( {\frac{{F_{n - m} F_{n + m} }}{\big(\sqrt[3]{F_n^2}  -  ( - 1)^{n - m} \sqrt[3]{F_m^2} \big)^3 }} \bigg),\quad n>m,	
	\end{split}\\
	\begin{split}
	\sum_{k = 1}^\infty & \frac{( - 1)^{km}}{k^2 \binom{3k}{k}} \left( {\frac{{27F_{n + m} F_{n - m} }}{{L_m^2 F_n^2 }}} \right)^k \\
	&= 6\arctan ^2 \!\bigg( {\frac{\sqrt 3 \sqrt[3]{F_{n - m}}}{2 \sqrt[3]{F_{n + m}}  - ( - 1)^m \sqrt[3]{F_{n - m}}}} \bigg) - \frac{1}{2}\log ^2\! \bigg( {\frac{{L_m F_n}}{{\big(\sqrt[3]{F_{n + m}}  + ( - 1)^m \sqrt[3]{F_{n - m}} \big)^3 }}} \bigg),
	\end{split}\nonumber\\
	\begin{split}
	\sum_{k = 1}^\infty & \frac{( - 1)^{k(m - 1)}}{k^2 \binom{3k}k} \left( {\frac{{27F_{n + m} F_{n - m} }}{F_m^2 L_n^2}} \right)^k  \\
	& = 6\arctan ^2\! \bigg( {\frac{\sqrt 3\, \sqrt[3]{F_{n - m}}}{{2\sqrt[3]{F_{n + m}}  + ( - 1)^m \sqrt[3]{F_{n - m}} }}} \bigg) - \frac{1}{2}\log ^2\! \bigg( {\frac{{F_m L_n }}{\big(\sqrt[3]{F_{n + m}}  - ( - 1)^m \sqrt[3]{F_{n - m}} \big)^3 }} \bigg),\nonumber
	\end{split}\\
	\begin{split}
	\sum_{k = 1}^\infty&\frac{1}{{k^2 \binom{3k}k}} \left( \frac{27}{4}\frac{{F_{2m} F_{2n} }}{{F_{n + m}^2 }} \right)^k \\ 
	&= 6\arctan ^2\! \bigg(\frac{\sqrt 3\sqrt[3]{L_m F_n}}{{2\sqrt[3]{L_nF_m}  - \sqrt[3]{L_mF_n}}} \bigg) - \frac{1}{2}\log ^2\! \bigg( {\frac{2F_{n + m} }{\big(\sqrt[3]{L_nF_m}  + \sqrt[3]{L_mF_n} \big)^3 }} \bigg),\quad  \frac{L_n}{L_m}>\frac{F_n}{F_m},\nonumber
	\end{split}\\
	\begin{split}\label{ex1}
	\sum_{k = 1}^\infty&  \frac{( - 1)^{km}}{k^2 \binom{3k}k} \left( {\frac{27L_{n + m} L_{n - m} }{{L_m^2 L_n^2 }}} \right)^k \\
	&= 6\arctan ^2\! \bigg( {\frac{\sqrt 3\sqrt[3]{L_{n - m}}}{2\sqrt[3]{L_{n + m}}  - ( - 1)^m\sqrt[3]{L_{n - m}}}} \bigg) - \frac{1}{2}\log ^2\! \bigg( {\frac{{L_m L_n }}{{\big(\sqrt[3]{L_{n + m}}  + ( - 1)^m \sqrt[3]{L_{n - m}}\big)^3 }}} \bigg),
	\end{split}\\
	\begin{split}\label{ex2}
	\sum_{k = 1}^\infty&  \frac{( - 1)^{k(m - 1)}}{k^2 \binom{3k}k } \left( {\frac{{27L_{n + m} L_{n - m} }}{25F_m^2 F_n^2 }} \right)^k \\ 
	&= 6\arctan ^2\! \bigg( {\frac{\sqrt 3\sqrt[3]{L_{n - m}}}{2\sqrt[3]{L_{n + m}}  + ( - 1)^m \sqrt[3]{L_{n - m}}}} \bigg) - \frac{1}{2}\log ^2\! \bigg( {\frac{{5F_m F_n }}{{\big(\sqrt[3]{L_{n + m}}  - ( - 1)^m \sqrt[3]{L_{n - m}}\big)^3 }}} \bigg).
	\end{split}
	\end{align}
\end{theorem}
\begin{proof} 	Straightforward using identities \eqref{F3} to \eqref{F8} and identity \eqref{A}.
\end{proof}
\begin{example}
	Identities \eqref{ex1} and \eqref{ex2} yield 
	\begin{align*}
	\sum_{k = 1}^\infty \frac{( - 1)^{kn}}{k^2 \binom{3k}k} \left( {\frac{54L_{2n}}{L_n^4}} \right)^k & = 6\arctan ^2\! \bigg( \frac{\sqrt 3} {\sqrt[3]{4L_{2n}}  - ( - 1)^n}\bigg) - \frac12\log ^2\! \bigg( {\frac{{L_n^2 }}{\big(\sqrt[3]{L_{2n}}  + ( - 1)^n \sqrt[3]{2}\big)^3 }}\bigg),\\
	\sum_{k = 1}^\infty \frac{( - 1)^{k(n-1)}}{k^2 \binom{3k}k} \left( {\frac{54L_{2n}}{25F_n^4}} \right)^k & = 6\arctan ^2\! \bigg( \frac{\sqrt 3} {\sqrt[3]{4L_{2n}}  + ( - 1)^n}\bigg) - \frac12 \log ^2\! \bigg( {\frac{{5F_n^2 }}{\big(\sqrt[3]{L_{2n}} -  ( - 1)^n \sqrt[3]{2}\big)^3 }}\bigg).
	\end{align*}
\end{example}

By writing $\cot^2 x$ for $x$ and setting $y=1$, a useful trigonometric version of identity \eqref{A} is obtained, namely,
\begin{equation}
\label{D}
\sum_{k = 1}^\infty  {\frac{{\big(\frac{27}{4}\big)^k }}{{k^2 \binom{3k}k}}\sin ^{2k} 2x}  = 6\arctan ^2\! \bigg( {\frac{{\sqrt 3 }}{{2\sqrt[3]{\cot^{2}x}  - 1}}} \bigg) - \frac{1}{2}\log ^2 \!\bigg( {\frac{{\csc ^2 x}}{{\big(\sqrt[3]{\cot^2 x}  + 1\big)^3 }}} \bigg).
\end{equation}
Identity \eqref{D} is valid for $x\in (0,\pi/4]$.

Evaluation of \eqref{D} at $x={\pi}/{12}$, $x={\pi}/{8}$ and $x={\pi}/{6}$, respectively, gives
\begin{align*}
\sum_{k = 1}^\infty  {\frac{{\big(\frac{27}{16}\big)^k }}{k^2 \binom{3k}k}} & = 6\arctan ^2 \!\bigg( {\frac{{\sqrt 3 }}{{2\sqrt[3]{7+4\sqrt 3}  - 1}}} \bigg) - \frac{1}{2}\log ^2\! \bigg( {\frac{4 + 2\sqrt 3 }{{\big(\sqrt[3]{7+4\sqrt 3 }  + 1\big)^3 }}} \bigg),
\end{align*}
\begin{align*}
\sum_{k = 1}^\infty  {\frac{{\big(\frac{27}{8}\big)^k }}{{k^2 \binom{3k}k}}}  &= 6\arctan ^2\! \bigg( {\frac{{\sqrt 3 }}{{2\sqrt[3]{3+2\sqrt 2}  - 1}}} \bigg) - \frac{1}{2}\log ^2\! \bigg( {\frac{{4 + 2\sqrt 2 }}{\big(\sqrt[3]{3+2\sqrt 2}  + 1\big)^3 }} \bigg),\\
\sum_{k = 1}^\infty  {\frac{{\big(\frac{81}{16}\big)^k }}{k^2 \binom{3k}k}} & = 6\arctan ^2 \!\bigg( {\frac{{\sqrt 3 }}{{2\sqrt[3]{3} - 1}}} \bigg) - \frac{1}{2}\log ^2\! \bigg( {\frac{4}{{(\sqrt[3]{3} + 1)^3}}} \bigg).
\end{align*}

Writing $-\cot^2x$ for $x$ and setting $y=1$ in \eqref{A} and noting that $1-\cot^2x=-{\cos 2x}\,{\sin^{-2}x}$,
we obtain another useful trigonometric version of \eqref{A}:
\begin{equation}\label{E}
\sum_{k = 1}^\infty  \frac{{\left(-\frac{27}{4}\right)^k }}{{k^2 \binom{3k}k}}\tan ^{2k} 2x  = 6\arctan ^2\! \bigg( {\frac{{\sqrt 3 }}{2\sqrt[3]{\cot^2 x}  + 1}} \bigg) - \frac{1}{2}\log ^2\! \bigg(\frac{{\csc ^2 x\,\cos 2x}}{{\big(\sqrt[3]{\cot^2x}  - 1\big)^3 }} \bigg),
\end{equation}
valid for $x\in (0,{\pi}/{8}]$. At $x={\pi}/{12}$ in \eqref{E} we obtain
\begin{align*}
\sum_{k = 1}^\infty  {\frac{{\left(-\frac94\right)^k }}{{k^2 \binom{3k}k}}} = 6\arctan^2\! \bigg( {\frac{{\sqrt 3 }}{{2\sqrt[3]{7 + 4\sqrt 3}  + 1}}} \bigg) - \frac{1}{2}\log ^2\! \bigg( {\frac{6+4\sqrt 3}{\big(\sqrt[3]{7 + 4\sqrt 3}  - 1\big)^3 }} \bigg).
\end{align*}

\subsection{Results from identity \eqref{B}}

\begin{theorem}\label{Th4}
	If $r$ is a positive integer, then
	\begin{align}\label{eq.grrm4zm}
	\begin{split}
	\sum_{k = 1}^\infty  {\frac{{( - 1)^{k(r - 1)} \big(\frac{27}{5}\big)^k }}{{k  \binom{3k}k F_r^{2k}}}}  &= 
	\frac{1}{\sqrt[3]{\alpha^r}L_r}
	\left(2\sqrt 3\big(\sqrt[3]{\alpha^{2r}}-(- 1)^r\big)\arctan\! \bigg( {\frac{\sqrt 3}{{2\sqrt[3]{\alpha^{2r}}  + ( - 1)^r }}} \bigg)\right.\\
	&\left.\,\quad- ( - 1)^{r}\big(\sqrt[3]{\alpha^{2r}}  + ( - 1)^r\big)\log\! \bigg( {\frac{{\sqrt 5\,\alpha ^r F_r  }}{{\big(\sqrt[3]{\alpha ^{2r}}  - ( - 1)^r \big)^3 }}} \bigg)\!\right)
	\end{split}
	\end{align}
	and
	\begin{align}
	\label{eq.nrc731g}
	\begin{split}
	\sum_{k = 1}^\infty  {\frac{{( - 1)^{kr} 27^k }}{k\binom{3k}kL_r^{2k} }}  &= \frac{\sqrt5}{5\sqrt[3]{\alpha^r}F_r}
	\left(2\sqrt 3\, \big(\sqrt[3]{\alpha ^{2r}}  + ( - 1)^r \big) \arctan\! \bigg( {\frac{\sqrt 3 }{2\sqrt[3]{\alpha ^{2r}}  - ( - 1)^r }} \bigg)\right.\\
	&\left.\,\quad + ( - 1)^r\big( \sqrt[3]{\alpha ^{2r}} - ( - 1)^r\big)\log\! \bigg( {\frac{{\alpha ^r L_r }}{\big(\sqrt[3]{\alpha ^{2r}}  + ( - 1)^r \big)^3 }} \bigg)\!\right)\!.
	\end{split}
	\end{align}
\end{theorem}
\begin{proof}
	Identity \eqref{eq.grrm4zm} is proved by setting $x=\alpha^{2r}$, $y=(-1)^{r + 1}$ in \eqref{B} and making use of identity \eqref{F1}. Identity \eqref{eq.nrc731g} follows from setting $x=\alpha^{2r}$, $y=(-1)^r$ in \eqref{B} and using  \eqref{F2}.
\end{proof}
\begin{example}
	Evaluation of \eqref{eq.grrm4zm} at $r=1$, $r=2$, $r=3$ and \eqref{eq.nrc731g} at $r=2$ and $r=3$, respectively, gives
	\begin{align*}
	\sum_{k = 1}^\infty  {\frac{{\big(\frac{27}{5}\big)^k }}{k\binom{3k}k}}  & = 2\sqrt3\, \frac{{\sqrt[3]{\alpha ^2}  + 1}}{\sqrt[3]{\alpha}}\arctan\! \bigg( {\frac{{\sqrt 3 }}{{2\sqrt[3]{\alpha ^2}  - 1}}} \bigg) + \frac{{\sqrt[3]{\alpha^2}  - 1}}{\sqrt[3]{\alpha}}\log\! \bigg( {\frac{\alpha \sqrt 5}{(\sqrt[3]{\alpha^2} + 1)^3}} \bigg),\\
	\sum_{k = 1}^\infty  {\frac{{ \left(-\frac{27}{5}\right)^k }}{k\binom{3k}k}}  &= \frac{2\sqrt3}{3} \frac{{\sqrt[3]{\alpha ^4}  - 1}}{\sqrt[3]{\alpha ^2} }\arctan\! \bigg( {\frac{{\sqrt 3 }}{{2\sqrt[3]{\alpha ^4}  + 1}}} \bigg) - \frac{{\sqrt[3]{\alpha^4}  + 1}}{3\sqrt[3]{\alpha^2}}\log\! \bigg( {\frac{{\alpha ^2 \sqrt 5 }}{{(\sqrt[3]{\alpha^4}  - 1)^3 }}} \bigg),\\
		\sum_{k = 1}^\infty  {\frac{{\big(\frac{27}{20}\big)^k }}{{k\binom{3k}k}}}  &= \frac{{\sqrt {15} }}{2}\arctan\! \big(\sqrt{15}-\sqrt{12}\big) - \frac{1}{4}\log\! \Big(\frac52\Big),\\
	\sum_{k = 1}^\infty  {\frac{{3^k }}{k\binom{3k}k}}  &= \frac{{2\sqrt {15} }}{5}\frac{{\sqrt[3]{\alpha^4}  + 1}}{\sqrt[3]{\alpha^2} }\arctan\! \bigg( {\frac{{\sqrt 3 }}{{2\sqrt[3]{\alpha^4} - 1}}} \bigg) + \frac{{\sqrt[3]{\alpha^4}  - 1}}{{\sqrt5\sqrt[3]{\alpha ^2} }}\log\! \bigg( {\frac{{3\alpha ^2 }}{{(\sqrt[3]{\alpha ^4}  + 1)^3 }}} \bigg),\\
	\sum_{k = 1}^\infty  {\frac{{\left(-\frac{27}{16}\right)^k }}{{k\binom{3k}k}}}  &= \frac{\sqrt {15} }{5}\arctan\! \bigg( \frac{4\sqrt 3-\sqrt{15}}{11}\bigg) - \log 2.
	\end{align*}
\end{example}
\begin{corollary}
	If $r$ is a positive integer, then
	\begin{align*}
	\sum_{k = 1}^\infty  {\frac{( - 1)^{(r-1)k}\left(\frac{27}{5}\right)^k}{{k  \binom{3k}kF_{3r}^{2k}}}} & = 2\sqrt {15}\, \frac{{F_r }}{{L_{3r} }}\arctan\! \bigg( {\frac{{\sqrt 3 }}{\alpha ^r (\alpha ^r  + L_r )}} \bigg) - ( - 1)^{r} \frac{{L_r }}{{L_{3r} }}\log\! \bigg( {\frac{{F_{3r} }}{{5F_r^3 }}} \bigg),\\
	\sum_{k = 1}^\infty  {\frac{( - 1)^{rk} 27^k }{k\binom{3k}kL_{3r}^{2k} }} & = \frac{2\sqrt {15}}{5} \frac{L_r }{F_{3r}}\arctan\! \bigg( {\frac{\sqrt 3 }{\alpha ^r (\alpha ^r  +  \sqrt 5F_r )}} \bigg)+ ( - 1)^r \frac{{F_r }}{{F_{3r} }}\log\! \bigg( {\frac{{L_{3r} }}{{L_r^3 }}} \bigg).
	\end{align*}
\end{corollary}

Replacing $x$ with $\cot^2x$ and setting $y=1$ in identity \eqref{B} gives
\begin{equation}\label{F}
\begin{split}
\sum_{k = 1}^\infty  {\frac{{\big(\frac{27}{4}\big)^k }}{{k\binom{3k}k}}\sin ^{2k}2x} & =  \frac{{2\sqrt 3\sin ^2 x}}{\cos 2x}\sqrt[3]{\cot^2x} \,\big(\sqrt[3]{\cot^2x}  + 1\big)\arctan\! \bigg( {\frac{{\sqrt 3 }}{2\sqrt[3]{\cot^2 x}  - 1}} \bigg)\\
&\quad + \frac{{\sin ^2 x}}{\cos 2x}\sqrt[3]{\cot^2 x}\, \big(\sqrt[3]{\cot^2 x}  - 1\big)\log\! \bigg( {\frac{{\csc ^2 x}}{\big(\sqrt[3]{\cot^2 x}  + 1\big)^3 }} \bigg),
\end{split}
\end{equation}
valid for $x\in (0,{\pi}/{4})$. 

At $x={\pi}/{12}$, $x={\pi}/{8}$ and $x={\pi}/{6}$, from \eqref{F} we have
\begin{align*}
\begin{split}
\sum_{k = 1}^\infty  \frac{{\big(\frac{27}{16}\big)^k }}{{k\binom{3k}k}}& = \Big(\sqrt[3]{2+\sqrt 3}+\sqrt[3]{2-\sqrt 3}\Big)\arctan\! \bigg(\frac{{\sqrt 3 }}{2\sqrt[3]{7+4\sqrt 3}  - 1} \bigg)\\
&\quad + \frac{\sqrt 3}{6}\Big(\sqrt[3]{2+\sqrt 3}-\sqrt[3]{2-\sqrt 3}\Big)\log\! \bigg( {\frac{8+4\sqrt 3}{{\big(\sqrt[3]{7+4\sqrt 3}  + 1\big)^3}}} \bigg),
\end{split}\\
\end{align*}
\begin{align*}
\begin{split}
\sum_{k = 1}^\infty  {\frac{{\big(\frac{27}{8}\big)^k }}{k\binom{3k}k}}&= \sqrt 3\, \Big(\sqrt[3]{1+\sqrt 2}-\sqrt[3]{1-\sqrt 2}\Big)\arctan\! \bigg( {\frac{{\sqrt 3 }}{2\sqrt[3]{3+2\sqrt 2}  - 1}} \bigg)\\
&\quad\,+ \frac{1}{2}\Big(\sqrt[3]{1+\sqrt 2} + \sqrt[3]{1-\sqrt 2}\Big)\log\! \bigg( \frac{4+2\sqrt 2}{\big(\sqrt[3]{3+2\sqrt 2}  + 1\big)^3} \bigg),
\end{split}\\
\begin{split}
\sum_{k = 1}^\infty  {\frac{{\big(\frac{81}{16}\big)^k }}{k\binom{3k}k}} & = \sqrt[6]{3^5}\, \big(\sqrt[3]{3} + 1\big)\arctan\! \bigg( {\frac{{\sqrt 3 }}{2\sqrt[3]{3} - 1}} \bigg)+ \frac{\sqrt[3]{3}\,(\sqrt[3]{3} - 1)}{2}\log\! \bigg( {\frac{4}{(\sqrt[3]{3} + 1)^3 }} \bigg).
\end{split}
\end{align*}

\subsection{Results from identity \eqref{C}}

\begin{theorem}
	If $r$ is a positive integer, then
	\begin{align}\label{Th2.6-Fib}
	\begin{split}
	\sum_{k = 1}^\infty & {\frac{{( - 1)^{(k-1)(r - 1)} \big(\frac{27}{5}\big)^k }}{\binom{3k}kF_r^{2k} }} =
	\frac{4}{L_r^2} \\
	&\!\!\!\!\!+ \frac{2\sqrt{15}}{3}\frac{F_r}{L^3_r}
	\!\left(
	2\,\big(\sqrt[3]{\alpha^{2r}}+\sqrt[3]{\beta^{2r}}\big) -(-1)^r\big(\sqrt[3]{\alpha^{4r}}+\sqrt[3]{\beta^{4r}}\big)\right)
	\!\arctan\! \bigg( {\frac{\sqrt 3}{2\sqrt[3]{\alpha ^{2r}}  + ( - 1)^r }}\bigg)	\\
	&\!\!\!\!\!+ \frac{( - 1)^r\sqrt5}{3}\frac{F_r}{L_r^3}
	\!\left(2\big(\sqrt[3]{\alpha^{2r}}-\sqrt[3]{\beta^{2r}}\big) +(-1)^r\big(\sqrt[3]{\alpha^{4r}}-\sqrt[3]{\beta^{4r}}\big)\right)
	\!\log\!\bigg(\frac{\sqrt5\alpha^rF_r}{\big(\sqrt[3]{\alpha^{2r}}-(-1)^r\big)^3}\bigg),
	\end{split}\\
	\begin{split}\label{Th2.6-Luc}
	\sum_{k = 1}^\infty & {\frac{( - 1)^{r(k-1)} 27^k }{\binom{3k}kL_r^{2k}}}=
	\frac{4}{5F_r^2}\\
	&\!\!\!\!+ \frac{2\sqrt{15}}{75}\frac{L_r}{F^3_r}
	\left(
	2\big(\sqrt[3]{\alpha^{2r}}+\sqrt[3]{\beta^{2r}}\big) +(-1)^r\big(\sqrt[3]{\alpha^{4r}}+\sqrt[3]{\beta^{4r}}\big)\right)
	\arctan\! \bigg( {\frac{\sqrt 3}{2\sqrt[3]{\alpha ^{2r}}  - ( - 1)^r }}\bigg)	\\
	&\!\!\!\! - \frac{( - 1)^{r}\sqrt5}{75}\frac{L_r}{F_r^3}
	\left(
	2\big(\sqrt[3]{\alpha^{2r}}-\sqrt[3]{\beta^{2r}}\big) - (-1)^r\big(\sqrt[3]{\alpha^{4r}} - \sqrt[3]{\beta^{4r}}\big)\right)
	\log\!\bigg(\frac{L_r}{\big(\sqrt[3]{\alpha^r}+\sqrt[3]{\beta^r}\big)^3} \bigg).
	\end{split}
	\end{align}
\end{theorem}
\begin{proof}
	Identities \eqref{Th2.6-Fib} and \eqref{Th2.6-Luc} are proved by setting, respectively, $x=\alpha^{r}$, $y=-\beta^{r}$ and $x=\alpha^{2r}$, $y=(-1)^r$ in \eqref{C} and making use of \eqref{F1} and \eqref{F2}.
\end{proof}
\begin{example}Evaluation of \eqref{Th2.6-Fib} at $r=1,2,3$ and \eqref{Th2.6-Luc} at $r=2,3,6$, respectively, gives
	\begin{align*}
	\sum_{k = 1}^\infty\frac{{\big(\frac{27}{5}\big)^k }}{\binom{3k}k} & = 4 + \frac{2\sqrt{15}}{3}\Big(2\,\big(\sqrt[3]{\alpha ^2}  + \sqrt[3]{\beta^2}\big)+\sqrt[3]{\alpha^4}+\sqrt[3]{\beta^4}\Big)\arctan\! \bigg( {\frac{\sqrt 3 }{2\sqrt[3]{\alpha ^2}  - 1}} \bigg) \\
	&\quad -\frac{\sqrt5}{3}\Big(2\,\big(\sqrt[3]{\alpha^2}  - \sqrt[3]{\beta^2}\big)-\big(\sqrt[3]{\alpha^4}-\sqrt[3]{\beta^4}\big)\Big)
	\log\! \bigg(\frac{\sqrt5\alpha}{(\sqrt[3]{\alpha^2} + 1 )^3}\bigg),\\
	\sum_{k = 1}^\infty  {\frac{{\left(-\frac{27}{5}\right)^k }}{\binom{3k}k}}  & = -\frac{4}{9}- \frac{2\sqrt{15}}{81}\left(2\big(\sqrt[3]{\alpha^4}  + \sqrt[3]{\beta^4}\big)-\big(\sqrt[3]{\alpha^8}+ \sqrt[3]{\beta^8}\big)\right)\arctan\!\bigg( {\frac{{\sqrt 3 }}{2\sqrt[3]{\alpha^4}  + 1}} \bigg) \\
	&\quad - \frac{\sqrt5}{81}\left(2\big(\sqrt[3]{\alpha ^4}  - \sqrt[3]{\beta^4}\big)+\sqrt[3]{\alpha ^8}-\sqrt[3]{\beta^8}\right)
	\log\! \bigg(\frac{\sqrt 5 \alpha^2}{(\sqrt[3]{\alpha^4} - 1)^3}\bigg),
	\end{align*}
	\begin{align*}
	\sum_{k = 1}^\infty  {\frac{{\big(\frac{27}{20}\big)^k }}{\binom{3k}k}}  & = \frac{1}{4} + \frac{13\sqrt{15}}{48}\arctan\! \big(\sqrt{15}-\sqrt{12}\big)-\frac{5}{96}
	\log\Big(\frac52\Big)
	\end{align*}
	and
	\begin{align*}
	\sum_{k = 1}^\infty  \frac{3^k }{\binom{3k}k} &= \frac{4}{5} +\frac{2\sqrt{15}}{25}\left(2\big(\sqrt[3]{\alpha^4}+\sqrt[3]{\beta^4}\big)+\sqrt[3]{\alpha^8}+\sqrt[3]{\beta^8}\right)   \arctan\! \bigg(\frac{\sqrt3}{2\sqrt[3]{\alpha^4}-1}\bigg) \\
	&\quad+ \frac{\sqrt5}{25}\left(2\big(\sqrt[3]{\alpha^4}-\sqrt[3]{\beta^4}\big)- \big(\sqrt[3]{\alpha^8}-\sqrt[3]{\beta^8}\big)\right)\log\! \big(1+\sqrt[3]{\alpha^2}+\sqrt[3]{\beta^2}\big),\\
	\sum_{k = 1}^\infty  {\frac{{\left(-\frac{27}{16}\right)^k }}{\binom{3k}k}}  & = -\frac{1}{5} + \frac{\sqrt{15}}{75}\arctan\! \bigg(\frac{4\sqrt3-\sqrt{15}}{11}\bigg)-\frac{1}{6}
	\log4,\\
	\sum_{k = 1}^\infty  \frac{\left(\frac{1}{12}\right)^k}{\binom{3k}{k}}  &=\frac{1}{80} +\frac{183\sqrt {15}}{3200}\arctan\! \bigg( \frac{\sqrt {15}-\sqrt{12}}{3}\bigg) -\frac{9}{256} \log\!\Big(\frac32\Big).
	\end{align*}
\end{example}

\section{Fibonacci and Lucas series involving inverses of the binomial coefficients $3n\choose n$}

In this section we will derive Fibonacci and Lucas identities which contain reciprocals of the binomial coefficients $\binom{3n}n$.
\begin{lemma}{\it\rm \cite{koshy}}
	If $p$ and $q$ are integers, then
	\begin{align}\label{eq.v57sl4d}
	F_p \alpha ^q  - F_{p + q} & =  - \beta ^p F_q ,
	\\[4pt]
	\label{eq.qkdldvh}
	F_{p + q}  - \beta ^q F_p  &= \alpha ^p F_q.
	\end{align}
\end{lemma}

\subsection{Fibonacci series associated with identity \eqref{A}}

\begin{theorem}\label{thm.fmxmjxw}
	Let $p$ and $q$ be integers such that $p\le -2$, $q\ge 4$ with $q>|p|+1$. Then
	\begin{align}\label{eq.dz9pgfw}
	\begin{split}
	\sum_{k = 1}^\infty&   \Big(\frac{-27F_p F_{p + q}}{F_q^2} \Big)^{\!k}\frac{F_{(2p + q)k}}{k^2 \binom{3k}{k}} \\
	&= \frac{6}{\sqrt 5} \Bigg( {\arctan ^2\! \bigg( \frac{{\sqrt 3\,\sqrt[3]{F_{p + q}}  }}{2\sqrt[3]{\alpha^q F_p}  + \sqrt[3]{F_{p + q}} } \bigg) - \arctan ^2\! \bigg( \frac{\sqrt 3 \sqrt[3]{F_p}}{{2\sqrt[3]{\alpha^qF_{p + q}}  +(-1)^q \sqrt[3]{F_p} }} \bigg)} \!\Bigg)\\
	&\quad - \frac{\sqrt 5}{10}\Bigg( {\log ^2\! \bigg( {\frac{(-1)^pF_q}{ \big(\sqrt[3]{\alpha^pF_{p+q}}  - \sqrt[3]{\alpha^{p+q}F_{p}} \big)^3 }}\bigg) - \log^2\! \bigg( {\frac{{\alpha ^{p + q} F_q }}{\big(\sqrt[3]{\alpha ^q F_{p + q}}  +(-1)^q \sqrt[3]{F_p}\big)^3 }} \bigg )} \!\Bigg),
	\end{split}
	\end{align}
	\begin{align}
	\begin{split}\label{eq.wwh3qpx}
	\sum_{k = 1}^\infty & \Big(\frac{- 27F_pF_{p + q}}{F_q^2} \Big)^k \frac{L_{(2p + q)k}}{k^2\binom{3k}{k}} \\
	&= 6\Bigg( {\arctan ^2\! \bigg( \frac{\sqrt 3\,\sqrt[3]{F_{p + q}}}{2\sqrt[3]{\alpha^q F_p}  + \sqrt[3]{F_{p + q}} } \bigg) + \arctan ^2\! \bigg( \frac{\sqrt 3 \sqrt[3]{F_p}}{2\sqrt[3]{\alpha^qF_{p + q}}  +(-1)^q \sqrt[3]{F_p} } \bigg)} \!\Bigg)\\
	&\quad - \frac{1}{2}
	\Bigg( {\log ^2\! \bigg( {\frac{(-1)^p F_q}{\big(\sqrt[3]{\alpha^pF_{p+q}}  - \sqrt[3]{\alpha^{p+q}F_{p}} \big)^3 }}\bigg) + \log^2\! \bigg( {\frac{{\alpha ^{p + q} F_q }}{\big(\sqrt[3]{\alpha ^q F_{p + q}}  -(-1)^q \sqrt[3]{F_p}\big)^3 }} \bigg )} \!\Bigg).
	\end{split}
	\end{align}
\end{theorem}
\begin{proof}
	Set $(x,y) = (F_p \alpha ^q, - F_{p + q} )$ in identity \eqref{A} and use \eqref{eq.v57sl4d} to obtain
	\begin{equation}\label{eq.du3gdpb}
	\begin{split}
	\sum_{k = 1}^\infty& \Big(\frac{- 27F_p F_{p + q}}{F_q^2} \Big)^{\!k}   {\frac{\alpha ^{(2p + q)k}}{{k^2 \binom{3k}k}} }\\
	&\quad = 6\arctan ^2\! \bigg( \frac{\sqrt 3 \sqrt[3]{F_{p + q}} }{{2\sqrt[3]{\alpha ^qF_p}  + \sqrt[3]{F_{p + q}} }} \bigg) - \frac{1}{2}\log ^2\! \bigg(\frac{(-1)^{p}F_q}{\big(\sqrt[3]{\alpha^pF_{p+q}}  - \sqrt[3]{\alpha^{p+q}F_{p}}\big)^3 }\bigg).
	\end{split}
	\end{equation}
	
	Similarly, $(x,y) = (F_{p + q} , - \beta ^q F_p )$ in identity \eqref{A} and the use of \eqref{eq.qkdldvh} gives
	\begin{equation}\label{eq.i0h0a2h}
	\begin{split}
	\sum_{k = 1}^\infty & \Big(\frac{- 27F_p F_{p + q}}{F_q^2} \Big)^{\!k}   {\frac{\beta^{(2p + q)k}}{{k^2 \binom{3k}k}} } \\
	&\quad=
	6\arctan ^2\! \bigg( \frac{\sqrt 3 \sqrt[3]{F_{p}} }{{2\sqrt[3]{\alpha ^qF_{p+q}} +(-1)^q\sqrt[3]{F_{p}} }} \bigg) - \frac{1}{2}\log ^2\! \bigg(\frac{\alpha^{p+q}F_q}{\big(\sqrt[3]{\alpha^qF_{p+q}}  -(-1)^q \sqrt[3]{F_{p}}\big)^3 }\bigg).
	\end{split}
	\end{equation}
	
	Identities \eqref{eq.dz9pgfw} and \eqref{eq.wwh3qpx} follow from the subtraction and addition of \eqref{eq.du3gdpb} and \eqref{eq.i0h0a2h} with the use of the Binet formulas \eqref{binet}.
\end{proof}
\begin{example} At $p=-2$ and $q=5$ from \eqref{eq.dz9pgfw} and \eqref{eq.wwh3qpx} we have  the following series:
	\begin{align*}
	\sum_{k = 1}^\infty    \frac{\left(\frac{54}{25} \right)^{k}\!F_k  }{{k^2 \binom{3k}k}} &=
	\frac{6\sqrt5}{5}\left(\arctan ^2\! \bigg( \frac{\sqrt 3 \sqrt[3]{2}}{{2\sqrt[3]{\alpha^5}} -\sqrt[3]{2} } \bigg) - 
	\arctan ^2\! \bigg( \frac{\sqrt 3}{{2\sqrt[3]{2\alpha^5}}+1} \bigg)\!\right)\\
	&\quad+
	\frac{\sqrt5}{10}\left(\log ^2\! \bigg(\frac{5\alpha^{3}}{\big(\sqrt[3]{2\alpha^5}  -1\big)^3 }\bigg)-\log ^2\! \bigg(\frac{5\alpha^{2}}{\big(\sqrt[3]{\alpha^5}  + \sqrt[3]{2}\big)^3 }\bigg)\!\right)\!,\\
	\sum_{k = 1}^\infty    \frac{\left(\frac{54}{25} \right)^{k}\!L_k }{{k^2 \binom{3k}k}} &=
	{6}\left(\arctan ^2\! \bigg( \frac{\sqrt 3 \sqrt[3]{2}}{{2\sqrt[3]{\alpha^5}} -\sqrt[3]{2} } \bigg) + 
	\arctan ^2\! \bigg( \frac{\sqrt 3}{{2\sqrt[3]{2\alpha^5}}+1} \bigg)\!\right)\\
	&\quad -
	\frac{1}{2}\left(\log ^2\! \bigg(\frac{5\alpha^{3}}{\big(\sqrt[3]{2\alpha^5}  -1\big)^3 }\bigg) + \log ^2\! \bigg(\frac{5\alpha^{2}}{\big(\sqrt[3]{\alpha^5}  + \sqrt[3]{2}\big)^3 }\bigg)\!\right)\!.
	\end{align*}
\end{example}

\subsection{Fibonacci series associated with identity \eqref{B}}

\begin{theorem}\label{Th9}
	Let $p$ and $q$ be integers such that $p\le -2$, $q\ge4$, and $q>|p|+1$. Then   
	\begin{align*}
	\begin{split}
	\!\!\!\!\!\!\!\!\!\frac{\sqrt5}{\sqrt[3]{F_pF_{p+q}}}&\sum_{k = 1}^\infty \Big(\frac{-27F_p F_{p + q}}{F_q^2} \Big)^{\!k} \frac{F_{(2p + q)k}}{k\binom{3k}{k}} \\
	&\!\!\!\!\!\!\!\!\!\!\!\!\!  = {2\sqrt{3}}\Bigg(\!A^{-}_{\alpha} \arctan\! \bigg( \frac{\sqrt 3\,\sqrt[3]{F_{p + q}}}{2\sqrt[3]{\alpha^q F_p}  + \sqrt[3]{F_{p + q}} } \bigg) + A^{-}_{\beta} \arctan\! \bigg( \frac{\sqrt 3\sqrt[3]{F_p}}{{2(-1)^q\sqrt[3]{\alpha^qF_{p+q}} + \sqrt[3]{F_{p}} }} \bigg)\!\Bigg)\\
	&\!\!\!\!\!\!\!\!\!\!\!\!\! \quad -\Bigg(\! A^{+}_{\alpha}\log\!\bigg( {\frac{\beta^pF_q}{\big(\sqrt[3]{F_{p + q}}- \sqrt[3]{\alpha^q F_p}  \big)^3 }}\bigg) - A^+_{\beta}\log\!\bigg( \frac{\alpha^p F_q}{\big(\sqrt[3]{F_{p + q}}  - \sqrt[3]{\beta^qF_p}\big)^3} \bigg)\! \Bigg),
	\end{split}\\
	\begin{split}
	\frac{1}{\sqrt[3]{F_pF_{p+q}}}\sum_{k = 1}^\infty& \Big(\frac{-27F_p F_{p + q}}{F_q^2} \Big)^{\!k} \frac{L_{(2p + q)k}}{k\binom{3k}{k}}\\ 
	&\!\!\!\!\!\!\!\!\!\!\!\!\!  = 2\sqrt 3\, \Bigg(\!A^{-}_{\alpha} \arctan\! \bigg( \frac{\sqrt 3\,\sqrt[3]{F_{p + q}}}{2\sqrt[3]{\alpha^q F_p}  + \sqrt[3]{F_{p + q}} } \bigg) -A^{-}_{\beta} \arctan\! \bigg( \frac{\sqrt 3 \sqrt[3]{F_p}}{{2(-1)^q\sqrt[3]{\alpha^qF_{p+q}} + \sqrt[3]{F_{p}} }} \bigg)\! \Bigg)\\
	&\!\!\!\!\!\!\!\!\!\!\!\!\!\quad - \Bigg(\! A^{+}_{\alpha}\log\!\bigg( {\frac{\beta^pF_q}{\big(\sqrt[3]{F_{p + q}}- \sqrt[3]{\alpha^q F_p}  \big)^3 }}\bigg) + A^+_{\beta}\log\!\bigg( \frac{\alpha^p F_q}{\big(\sqrt[3]{F_{p + q}}  - \sqrt[3]{\beta^qF_p}\big)^3} \bigg)\! \Bigg),
	\end{split}
	\end{align*}
	where 
	$$
	A^{\pm}_s=\sqrt[3]{s^q}\,\frac{\sqrt[3]{s^q F_p}\pm\sqrt[3]{F_{p+q}}}{s^qF_p+F_{p+q}}.
	$$
\end{theorem}
\begin{proof}
	The proof is similar to that one given for Theorem \ref{thm.fmxmjxw} and omitted.
\end{proof}
\begin{example} At $p=-2$ and $q=5$ from Theorem \ref{Th9} we obtain the following series:
	\begin{align*}
	\frac{1}{\sqrt[3]{2\alpha^5}}\sum_{k = 1}^\infty\frac{\left(\frac{54}{25} \right)^{k}F_k }{k \binom{3k}k} &\\
	&\!\!\!\!\!\!\!\!\!\!\!\!\!\!\!\!=
	\frac{2\sqrt{15}}{5}
	 \left(\frac{\sqrt[3]2+\sqrt[3]{\alpha^5}}{2-\alpha^5}\arctan\!\bigg(\frac{\sqrt3}{1-\sqrt[3]{4\alpha^5}}\bigg)+\frac{1-\sqrt[3]{2\alpha^5}}{1+2\alpha^5}\arctan\!\bigg( \frac{\sqrt 3}{2\sqrt[3]{2\alpha^5}+1} \bigg)\!\right)\\
	&\!\!\!\!\!\!\!\!\!\!\!\!\!\!\!\!\quad+
	\frac{\sqrt5}{5}
	\left(\frac{\sqrt[3]{2}-\sqrt[3]{\alpha^5}}{2-\alpha^5}
	\log\! \bigg(\frac{5\alpha^2}{\big(\sqrt[3]{2}+\sqrt[3]{\alpha^5}\big)^3}\bigg)+\frac{1+\sqrt[3]{2\alpha^5}}{1+2\alpha^5}\log\! \bigg(\frac{5\alpha^{3}}{\big(\sqrt[3]{2\alpha^5}  - 1\big)^3}\bigg)\!\right)\!,\\
	\frac{1}{\sqrt[3]{2\alpha^5}}\sum_{k = 1}^\infty     \frac{\left(\frac{54}{25} \right)^{k}L_k }{k \binom{3k}k} &\\
	&\!\!\!\!\!\!\!\!\!\!\!\!\!\!\!\!\!\!=
	2\sqrt{3} \left(\frac{\sqrt[3]2+\sqrt[3]{\alpha^5}}{2-\alpha^5}\arctan\!\bigg(\frac{\sqrt3}{1-\sqrt[3]{4\alpha^5}}\bigg)-\frac{1-\sqrt[3]{2\alpha^5}}{1+2\alpha^5}\arctan\!\bigg( \frac{\sqrt 3}{2\sqrt[3]{2\alpha^5}+1} \bigg)\!\right)\\
	&\!\!\!\!\!\!\!\!\!\!\!\!\!\!\!\!\!\!\quad+
	\frac{\sqrt[3]{2}-\sqrt[3]{\alpha^5}}{2-\alpha^5}
	\log\! \bigg(\frac{5\alpha^2}{\big(\sqrt[3]{2}+\sqrt[3]{\alpha^5}\big)^3}\bigg)-\frac{1+\sqrt[3]{2\alpha^5}}{1+2\alpha^5}\log\! \bigg(\frac{5\alpha^{3}}{\big(\sqrt[3]{2\alpha^5}  - 1\big)^3}\bigg).
	\end{align*}
\end{example}

\subsection{Fibonacci series associated with identity (C)}

\begin{theorem}\label{Th10}
	Let $p$ and $q$ be integers such that $p\le -2$, $q\ge4$, and $|q|>|p|+1$. Then
	\begin{align*}
	\begin{split}
	&\!\!\!\!\!\!\!\!\!\!\frac{1}{F_q\sqrt[3]{F_pF_{p+q}}}\sum_{k = 1}^\infty (-1)^{k-p}\Big(\frac{27F_p F_{p + q}}{F_q^2} \Big)^{\!k} \frac{F_{(2p + q)k}}{\binom{3k}{k}}\\
	& \!\!\!\!\!\!= \frac{4(-1)^{p-1}\sqrt[3]{F^2_{p+q}F^2_p}\big(F^2_{p+q}-(-1)^qF_p^2\big)}  {(F_{p+q}+\alpha^qF_p)^2(F_{p+q}+\beta^qF_p)^2} \\
	& \, - \frac{2\sqrt{15}}{15}\Bigg(B^{+}_{\alpha} \arctan\! \bigg( \frac{\sqrt 3\,\sqrt[3]{F_{p + q}}}{2\sqrt[3]{\alpha^q F_p}  + \sqrt[3]{F_{p + q}} } \bigg) + B^{+}_{\beta} \arctan\!\bigg( \frac{\sqrt 3\sqrt[3]{F_p}}{{2(-1)^q\sqrt[3]{\alpha^qF_{p+q}} + \sqrt[3]{F_{p}} }} \bigg) \!\Bigg)\\
	& \,+\frac{\sqrt5}{15} \Bigg(B^{-}_{\alpha}
	\log\!\bigg( {\frac{(-1)^p F_q}{\big(\sqrt[3]{\alpha^pF_{p + q}}- \sqrt[3]{\alpha^{p+q} F_p}  \big)^3 }}\bigg) -	 B^{-}_{\beta}\log\!\bigg( \frac{\alpha^{p+q} F_q}{\big(\sqrt[3]{\alpha^qF_{p + q}}  -(-1)^q \sqrt[3]{F_p}\big)^3} \bigg) \!\Bigg),
	\end{split}
	\end{align*}
	\begin{align*}
	\begin{split}
	&\!\!\!\!\!\!\!\!\!\!\frac{1}{F_q\sqrt[3]{F_pF_{p+q}}}\sum_{k = 1}^\infty (-1)^{k-p}\Big(\frac{27F_p F_{p + q}}{F_q^2} \Big)^{\!k} \frac{L_{(2p + q)k}}{\binom{3k}{k}}\\
	&=\frac{4(-1)^{p-q-1}\sqrt[3]{F^2_{p+q}F^2_p}}{F_q}\cdot  \frac{F^2_pL_q+(-1)^qF^2_{p+q}L_q+4F_pF_{p+q}}{(F_{p+q}+\alpha^qF_p)^2 (F_{p+q}+\beta^qF_p)^2}\\ 
	& \quad  - \frac{2}{\sqrt3} 
	\Bigg(B^{+}_{\alpha} \arctan\! \bigg( \frac{\sqrt 3\,\sqrt[3]{F_{p + q}}}{2\sqrt[3]{\alpha^q F_p}  + \sqrt[3]{F_{p + q}} } \bigg) - B^{+}_{\beta} \arctan \!\bigg( \frac{\sqrt 3 \sqrt[3]{F_p}}{{2(-1)^q\sqrt[3]{\alpha^qF_{p+q}} + \sqrt[3]{F_{p}} }} \bigg) \!\Bigg)\\
	&\quad + \frac{1}{3}\Bigg( B^{-}_{\alpha}\log\!\bigg( {\frac{(-1)^pF_q}{\big(\sqrt[3]{\alpha^pF_{p + q}}- \sqrt[3]{\alpha^{p+q} F_p}  \big)^3 }}\bigg) + B^{-}_{\beta}\log\!\bigg( \frac{\alpha^{p+q} F_q}{\big(\sqrt[3]{\alpha^qF_{p + q}}  -(-1)^q \sqrt[3]{	F_p}\big)^3} \bigg) \!\Bigg),
	\end{split}
	\end{align*}
	where 
	$$
	B^{\pm}_s=\frac{\sqrt[3]{s^{q-3p}}}{(s^qF_p + F_{p+q})^3}\bigg(\sqrt[3]{s^{4q}F_p^4}\pm\sqrt[3]{F^4_{p+q}}\mp 2\sqrt[3]{s^qF_p F_{p+q}}\,\Big(\sqrt[3]{s^{2q}F^2_p}\pm\sqrt[3]{F^2_{p+q}}\Big)\!\bigg).
	$$
\end{theorem}
\begin{proof} 	The proof is similar to the previous two proofs.	
\end{proof}
\begin{example} At $p=-2$ and $q=5$ from Theorem \ref{Th10}  we obtain the following series:
	\begin{align*}
	\sum_{k = 1}^\infty\frac{\left(\frac{54}{25} \right)^{k}\!F_k }{k \binom{3k}k} &= \frac{200}{361}+\frac{10\sqrt[3]{2\alpha^{11}}}{\sqrt{15}}
	\left(\frac{\sqrt[3]{16}(1+\alpha^5)+\sqrt[3]{\alpha^5}(4+\alpha^5)}{(5\alpha+1)^3} \arctan\!\bigg(\frac{\sqrt3}{2\sqrt[3]{\alpha^5/2}-1}\bigg)\right.\\
	&\left.\quad+\,
	\frac{\sqrt[3]{16\alpha^5}(\alpha^5-1)+1-4\alpha^5}
	{(\alpha-5)^3\alpha^{11}}\arctan\!\bigg( \frac{\sqrt 3}{2\sqrt[3]{2\alpha^5}+1} \bigg)\!\right)\\
	&\quad-
	\frac{\sqrt5\sqrt[3]{2\alpha^{11}}}{3}
	\left(\frac{\big(\sqrt[3]{2}+\sqrt[3]{\alpha^5}\big)\big(\sqrt[3]{2}-\sqrt[3]{\alpha^5}\big)^3}{(5\alpha+1)^3}
	\log\! \bigg(\frac{5\alpha^2}{\big(\sqrt[3]{\alpha^5}+\sqrt[3]{2}\big)^3}\bigg)\right.\\
	&\left.\quad+\frac{\big(\sqrt[3]{2\alpha^5}-1\big)\big(\sqrt[3]{2\alpha^5}+1\big)^3}{(\alpha-5)^3\alpha^{11}}\log\! \bigg(\frac{5\alpha^{3}}{\big(\sqrt[3]{2\alpha^5}  - 1\big)^3}\bigg)\!\right)\!,
		\end{align*}	
		\begin{align*}
	\sum_{k = 1}^\infty    \frac{\left(\frac{54}{25} \right)^{k}\!L_k }{k \binom{3k}k} &= \frac{328}{361}+\frac{10\sqrt{3}\sqrt[3]{2\alpha^{11}}}{3}
	\left(\frac{\sqrt[3]{\alpha^5}(\alpha^5+4)+\sqrt[3]{16}(\alpha^5+1)}{(\alpha^5-2)^3} \arctan\!\bigg(\frac{\sqrt3}{2\sqrt[3]{\alpha^5/2}-1}\bigg)\right.\\
	&\left.\quad+\,
	\alpha\frac{\sqrt[3]{16\alpha^{20}}+1-\sqrt[3]{16\alpha^5}\big(\sqrt[3]{4\alpha^{10}}+1\big)}{(2\alpha^5+1)^3}\arctan\!\bigg( \frac{\sqrt 3}{2\sqrt[3]{2\alpha^5}+1} \bigg)\!\right)\\
	&\quad-
	\frac{5\sqrt[3]{2\alpha^{11}}}{3}
	\left(\frac{\big(\sqrt[3]{2}+\sqrt[3]{\alpha^5}\big)\big(\sqrt[3]{2}-\sqrt[3]{\alpha^5}\big)^3}{(\alpha^5-2)^3}
	\log\! \bigg(\frac{5\alpha^2}{\big(\sqrt[3]{\alpha^5}+\sqrt[3]{2}\big)^3}\bigg)\right.\\
	&\left.\quad+\frac{\alpha\big(\sqrt[3]{2\alpha^5}-1\big)\big(\sqrt[3]{2\alpha^5}+1\big)^3}{(2\alpha^5+1)^3}\log\! \bigg(\frac{5\alpha^{3}}{\big(\sqrt[3]{2\alpha^5}  - 1\big)^3}\bigg)\!\right)\!.
	\end{align*}
\end{example}

\section{Concluding comments}

In this paper we presented new closed forms for some types of infinite series involving binomial coefficients $\binom{3n}{n}$. To prove our results, we applied some routine arguments, combining Batir's formula \eqref{Batir} with Binet's formulas for Fibonacci and Lucas numbers. Using similar techniques, we can establish series evaluations involving binomial coefficients $\binom{3n}{n}$ with Fibonacci and Lucas polynomials and other known number and polynomial sequences. 

Let us give, for example, a generalization of Theorems \ref{Th1} and \ref{Th4} to the case of the Horadam sequence defined by the recurrence 
$$W_n=pW_{n-1}-qW_{n-2},\quad n\geq2,$$ with initial values $W_0=a$ and $W_1=b$.

Let 
$$\Delta=\sqrt{p^2+4q},\,\,\, \alpha_*=(p+\Delta)/2,\,\,\, \beta_*=(p-\Delta)/2,\,\,\, A=b-a\beta_*,\,\,\, B=b-a\alpha_*.$$    

The following identities hold for positive integer $r$: 
\begin{align*}
\sum_{k = 1}^\infty  \frac{(-1)^{k(r-1)}}{k^2 \binom{3k}k}&\left(\frac{27AB q^r}{\Delta^2}\right)^{k}W^{-2k}_r \\
&= 
6\arctan^2\!\bigg(\frac{\sqrt3\sqrt[3]{Bq^r}}{2\sqrt[3]{A\alpha_*^{2r}}+\sqrt[3]{B(-q)^r}}\bigg)-\frac12
\log^2\! \bigg(\frac{\alpha_*^r\Delta W_r}{\big(\sqrt[3]{A\alpha_*^{2r}}-\sqrt[3]{B(-q)^r}\big)^3}\bigg)
\end{align*}
and
\begin{align*}
\frac{A\alpha_*^{2r}+B(-q)^r}{\sqrt[3]{AB\alpha_*^{2r}q^r}} &\sum_{k = 1}^\infty  \frac{(-1)^{k(r-1)}}{k \binom{3k}k}\left(\frac{27AB q^r}{\Delta^2}\right)^{k}W^{-2k}_r \\
&= 
2\sqrt3\big(\sqrt[3]{A\alpha_*^{2r}}-\sqrt[3]{B(-q)^{r}} \big)\arctan\!\bigg(\frac{\sqrt3\sqrt[3]{Bq^r}}{2\sqrt[3]{A\alpha_*^{2r}}+\sqrt[3]{B(-q)^r}}\bigg)\\
&\quad-(-1)^r
\big(\sqrt[3]{A\alpha_*^{2r}}+\sqrt[3]{B(-q)^{r}} \big)
\log\! \bigg(\frac{\alpha_*^r \Delta  W_r}{\big(\sqrt[3]{A\alpha_*^{2r}}-\sqrt[3]{B(-q)^r}\big)^3}\bigg).
\end{align*}

\section*{Acknowledgment}
We would like to thank Professor Wenchang Chu for drawing our attention to several papers related to the subject of our research.\\

\end{document}